\documentclass[12pt,a4paper]{amsart}
\usepackage{amsmath}
\usepackage{amsthm}
\usepackage{amssymb}
\usepackage{mathrsfs}
\usepackage[english]{babel}
\usepackage{ot1patch}

\newtheorem{thm}{Theorem}[section]
\newtheorem{lem}[thm]{Lemma}

\newtheorem{prop}[thm]{Proposition}

\theoremstyle{remark}
\newtheorem{rem}[thm]{Remark}
\newtheorem*{rem*}{Remark}

\theoremstyle{definition}
\newtheorem{dfn}[thm]{Definition}
\newtheorem{ex}[thm]{Example}

\numberwithin{equation}{section}

\newcommand{\Rz}{\mathbb{R}}

\newcommand{\M}{\mathcal{M}}

\begin{document}
\title[The Kuratowski convergence of medial axes\dots]{The Kuratowski convergence of medial axes and conflict sets}

\author{Adam Białożyt}
\address{Jagiellonian University, Faculty of Mathematics and Computer Science, Institute of Mathematics, \L ojasiewicza 6, 30-348 Krak\'ow, Poland, ORCID 0000-0001-8830-935X}\email{adam.bialozyt@uj.edu.pl}

\author{Anna Denkowska}
\address{Cracow University of Economics, Department of Mathematics, Rako\-wicka 27, 31-510 Krak\'ow, Poland, ORCID 0000-0003-4308-8180}
\email{anna.denkowska@uek.krakow.pl}

\author{Maciej P. Denkowski}
\address{Jagiellonian University, Faculty of Mathematics and Computer Science, Institute of Mathematics, \L ojasiewicza 6, 30-348 Krak\'ow, Poland, ORCID 0000-0001-7231-2482 (Corresponding author)}\email{maciej.denkowski@uj.edu.pl}

\date{February 24th 2015, revised: February 12th 2016, November 2nd 2016, May 11th 2020, October 22nd 2022}

\begin{abstract}
This paper consists of two parts. In the first one we study the behaviour of medial axes (skeletons) of  closed sets in a connected complete Riemannian manifold $\M$ under deformations. The second one is devoted to a similar study of conflict sets. We apply a new approach to the deformation process. Instead of seeing it as a `jump' from the initial to the final state, we perceive it as a continuous process, expressed using the Kuratowski convergence of sets (hence, unlike other authors, we do not require any regularity of the deformation). Our main `medial axis inner semi-continuity' result has already proved useful, as it was used to compute the tangent cone of the medial axis with application in singularity theory.
\end{abstract}
\maketitle


\section{Introduction}

It has been known for a long time that the \textit{medial axis} (\footnote{As explained later, what we mean by \textit{medial axis} of a given closed, nonempty set $X\subset \M$ is the set of points in $\M $ for which there is more than one distance realising geodesic to $X$ with respect to the induced distance.}) is highly unstable under small deformations. In particular, F. Chazal and R.~Soufflet gave in \cite{ChS} a simple illustration of this fact: the medial axis of a planar circle is its central point, but even the smallest `protuberance' on the circle leads to the medial axis becoming a whole segment. Their paper \cite{ChS} is entirely devoted to showing that under some  hypotheses on $X$ there is a kind of stability of the medial axis under $\mathcal{C}^2$ deformations. Their approach consists in  looking at the initial and the final steps only --- with nothing in between, so to say.

In the present paper we adopt another, natural, point of view: we see the deformation as a continuous process that we do not even require to be smooth. This lets us have some insight of what is happening to the medial axis. 

We have chosen the Kuratowski (or \textit{Painlev\'e-Kuratowski}) convergence of closed sets as a method of approach to this process of deformation. In this way we have obtained our first main Theorem \ref{main} which has already proved useful: it is used in \cite{BD}  in order to compute the tangent cone of the medial axis of a definable or subanalytic set, which permitted to give some relation between the medial axis and the type of singularity of the given definable set (see Example \ref{bd}).

After settling the question of the behaviour of the medial axis, we turn to studying the evolution of the conflict set of a family of sets, obtaining a similar semi-continuity result 
as well as a convergence statement, see Theorem \ref{CS}, the second main result of the paper.

\medskip
The paper is divided into two distinct parts and organized as follows. The first four sections are entirely devoted to medial axes. After recalling the basic definitions (Section \ref{2}), we give a large range of examples to illustrate the obstacles encountered by mathematicians, us included, who believed that the problem of convergence of the medial axes was an easy one (Section \ref{3}).  The examples lead to two natural conjectures that turn out to be false (as shown in Examples \ref{3.5} and \ref{3.11}, respectively). We do this to show why our approach via the Kuratowski convergence seems to be the most appropriated and the result obtained in Theorem \ref{main} is somehow optimal (cf. Remark \ref{Rk} and Example \ref{3.11}). Although natural, it has eluded to specialists who had judged it at first elementary and easy. We have also been through this experience and it is our way of sharing it with the reader. 

The last section concerns conflict sets. It contains the basic definitions, a brief discussion of the relation between conflict sets and medial axes (as well as Voronoi diagrams), some auxiliary results and finally the second main Theorem \ref{CS}.

\medskip
As a matter of fact, the first versions of the paper had dealt exclusively with definable (semi-algebraic, subanalytic) sets in ${\Rz}^n$ as it seemed that we would definitely need some control over the topology involved. However, thanks to the endeavours of the first-named author, it turned out that the result does not require any special `tame geometry' assumptions. As some of the auxiliary `definable' results may be of interest on their own, we gather them in the appendix.

\section{Acknowledgements}
The authors would like to express their gratitude to Lev Birbrair for his interest in the paper and many discussions on the subject.

During the preparation of the medial axis part of this paper the third-named author was partly supported by Polish Ministry of Science and Higher Education grant 1095/MOB/2013/0. He would like to thank the University of Lille 1 for hospitality.

\section{Preliminaries}\label{2}
Throughout the  paper we study the behaviour of the medial axes of a family of closed subsets of a {\it connected complete Riemannian manifold} $\M$. The medial axis is closely related to the notion of \textit{central set} (the set of centres of maximal balls contained in $\M \setminus X$) and cut locus, and appears sometimes under the name of skeleton or cut locus (although this need not denote precisely the same concept). It plays an important role in pattern recognition (see \cite{ChS} for references), but has applications also in variational analysis (historically \cite{Cl} seems to be the first paper hinting at that, see also \cite{F}) or singularity theory which is particularly of interest to us (cf. \cite{Y}, \cite{BS}, \cite{D}, \cite{BD}).

\subsection{Medial axis and central set}
We recall the basic definitions. For a given pair of points $p,q$ on a Riemannian manifold $(\M,g)$, using the metric tensor $g$, we can define the distance $d(p,q)$ between them as the infimum of lengths of all piecewise differentiable curves on $\M$ connecting $p$ to $q$ (\footnote{Mind that in case $(\M,g)=(\Rz^n,\mathbb{I})$, the process reconstructs the usual Euclidean distance.}). Then for a closed, nonempty set $X\subset \M$ and a point $p\in \M$ we shall write
$$ \mathrm{dist}(p,X):=\inf\{d(p,q)\mid q\in X\}.$$
We define the set of closest points to $p\in \M$ as 
$$
m(p):=\{x\in X\mid d(p,x)=\mathrm{dist}(p,X)\}
$$
which is a compact nonempty subset of $X$. Since we assume $\M$ to be complete, thanks to Hopf-Rinow Theorem any point of $x\in m(p)$ can be connected with $p$ via a geodesic $\gamma_{x,p}$ of length $\mathrm{dist}(p,X)$ originating from the point $x$. We shall write $\Omega_{X,p}$ for the set of geodesics of minimal length connecting a point in $m(p)$ with $p$ and $\gamma_{X,p}$ for an arbitrary geodesic in $\Omega_{X,p}$. Then the \textit{medial axis of $X$} is defined to be the set 
$$
M_X:=\{p\in \M \mid \exists \gamma,\tilde{\gamma}\in \Omega_{X,p}: \gamma\neq \tilde{\gamma} \}
$$
i.e. the set of points where the distance realising geodesic is not unique. Observe here, that in the Euclidean case it is equivalent to the multifunction $x\mapsto m(x)$ not being univalent at the given point. Thus, $M_X$ is a strict generalisation of Blum's medial axis to $\M$. A simple example of a point on a sphere together with its antipodal shows however, that the notion is more intricate in the Riemannian setting. 

Recall that a point on a geodesic $\gamma$ emanating from $q\in\M$ is called a \textit{ cut point of $q$ along $\gamma$} if it is the first point $\gamma(t_0)$ such that for all points $\gamma(t)$ with $t>t_0$ there exists a geodesic from $\gamma(t)$ to $q$ shorter than $\gamma$. The collection of cut points of $q$ along all possible geodesics is denoted by $Cut(q)$. It is an important notion in global analysis of Riemannian manifolds as it gives a kind of skeleton on which the manifold is spanned and inherits a number of its topological properties. Readers interested in learning more about the cut loci and their relation with conjugate points may consult the splendid survey of Shoshichi Kobayashi \cite{K}.

We can adapt the definition of the cut locus of a point to closed sets $X$ by simply considering the geodesics originating from points of $X$ and collecting the points after which the said geodesics cease to minimise the distance to $X$. More precisely, we say that a point $p\in\M$ belongs to the cut locus of a closed set $X\subset \M$, and denote this by $p\in Cut(X)$, if there exists a point $q\in X$ and a geodesic $\gamma$ originating from $q$ such that $p$ is the first point $\gamma(t_0)$ such that for all points $\gamma(t)$ with $t>t_0$ there exists a geodesic from $\gamma(t)$ to $X$ shorter than $\gamma$. The notion of cut locus of a closed set has become an object of recent interest, notably due to \cite{A} and \cite{TS}. We note here also that by the results of \cite{A} we have
$$ M_X\subset Cut(X)\subset \overline{M_X}$$
Indeed if a point $p$ admits at least two distinct minimising geodesics $\gamma,\tilde{\gamma} \in \Omega_{X,p}$ then neither of them is minimising beyond $p$ which proves the first inclusion. The second inclusion is a straightforward consequence of \cite[Theorem 1.4]{A}.


\subsection{Kuratowski convergence}

For set convergence we refer the reader the excellent book \cite{RW}. Here we adopt the point of view of \cite{DD} (slightly generalised). Hereafter, $\mathbb{B}(p,r)$ denotes the open ball of radius $r>0$ centred at $p$.

In what follows we will consider a set $X\subset \Pi\times\M$ in the variables $(t,x)$ where $\Pi$ is considered to be a $T_1$ topological space of parameters with a distinguished non-isolated point $0$ having a countable basis of neighbourhoods. We shall write $X_t$ for the $t$-\textit{sections} $X_t=\{x\in\M\mid (t,x)\in X\}$. Consider the projection $\pi(t,x)=t$ and assume that $0\in\overline{\pi(X)\setminus\{0\}}$. We define the \textit{upper} and the \textit{lower Kuratowski limits} of $X_t$ for $\pi(X)\setminus\{0\}\ni t\to 0$ as in \cite{DD}:\begin{itemize}
\item $x\in\limsup X_t$ iff for any neighbourhoods $U\ni x, V\ni 0$, there is a parameter $t\in\pi(X)\cap V\setminus\{0\}$ such that $X_t\cap U\neq\varnothing$;
\item $x\in\liminf X_t$ iff for any neighbourhood $U\ni x$, there is a neighbourhood $V\ni 0$ such that for any $t\in\pi(X)\cap V\setminus\{0\}$, we have $X_t\cap U\neq\varnothing$.
\end{itemize}
As $\liminf X_t\subset \limsup X_t$, we have convergence iff the converse inclusion holds. In particular, $X_t$ converges to $X_0$ when $t\to 0$, which we denote by $X_t\stackrel{K}{\longrightarrow} X_0$, iff $$\limsup X_t\subset X_0\subset\liminf X_t.$$ 

Recall that the upper and lower limits are always closed sets and do not change if we compute them for the closures $\overline{X_t}$.

It is worth noting that $X_0=\limsup X_t$ iff for any compact set $K$ disjoint with $X_0$, there is a neighbourhood $V$ of $0\in\Pi$ such that for all $t\in V$, $X_t\cap K=\varnothing$ (see e.g. \cite{DD}). 

The next general lemma can be found in \cite{DK}; here we give a slightly different and more straightforward proof.
\begin{lem}[cf. \cite{DK} Lemma 2.1]\label{ciqglosc}
Assume that $Y\subset\Pi\times\M$ has closed $t$-sections, $0\in\overline{\pi(Y)\setminus\{0\}}$ for $\pi(t,x)=t$ and $Y_t\stackrel{K}{\rightarrow} Y_0$. Then for any $y\in\M$, $d(x,Y_t)\to d(y,Y_0)$, when $\pi(Y)\times \M\ni (t,x)\to (0,y)$.
\end{lem}
\begin{proof}
Fix $y\in\M$ and put $d:=d(y,Y_0)$. There are three possibilities.

(1) $d=+\infty$ i.e. $Y_0=\varnothing$. From the definition of the convergence it follows directly that for any $R>0$, there is $Y_t\subset\M\setminus\mathbb{B}(y,R)$, for all $t$ sufficiently close to the 0. Therefore,  $d(x,Y_t)\to +\infty=d(y,Y_0)$, as required.

(2) $0<d<+\infty$. Then for any $\varepsilon>0$, $\mathbb{B}(a,d+\varepsilon)\cap Y_0\neq\varnothing$, while $\mathbb{B}(y,d-\varepsilon)\cap Y_0=\varnothing$ and by the convergence $Y_t\stackrel{K}{\rightarrow} Y_0$ these properties are shared by all the $Y_t$ for $t$ sufficiently close to $0$. Therefore, $d-\varepsilon<d(y,Y_t)<d+\varepsilon$ for $t\in\pi(Y)$ close to the 0. Recall that the distance $d(x,Y_t)$ is 1-Lipschitz, whence
$$
|d(x,Y_t)-d|\leq |d(x,Y_t)-d(y,Y_t)|+|d(y,Y_t)-d|\leq d(x,y)+\varepsilon
$$
which gives the convergence sought for.

(3) $d=0$. Then $d(x,Y_t)\leq |d(x,Y_t)-d(y,Y_t)|+d(y,Y_t)$ and the first term is bounded by $d(x,y)$ while $d(y,Y_t)\to 0$ when $t\to 0$. Indeed, since for any $\varepsilon>0$, $\mathbb{B}(y,\varepsilon)\cap Y_0\neq\varnothing$, this holds true also for $Y_t$ whenever $t$ is sufficiently close to the 0, whence $d(y,Y_t)\leq \varepsilon$. This ends the proof.
\end{proof}
\begin{rem}\label{granicecz}
It is useful to note that the Kuratowski limits $\liminf Y_t$ and $\limsup Y_t$ at a point $t_0\in\overline{\pi(Y)\setminus\{t_0\}}$ admitting a countable basis of neighbourhoods can be described also in the following convenient way:
\begin{align*}
&x\in \liminf Y_t\ \Leftrightarrow\ \forall \pi(Y)\setminus\{y_0\}\ni t_\nu\to t_0, \exists Y_{t_\nu}\ni x_\nu\to x;\\
&x\in \limsup Y_t\ \Leftrightarrow\ \exists \pi(Y)\setminus\{t_0\}\ni t_\nu\to t_0, \exists Y_{t_\nu}\ni x_\nu\to x.
\end{align*}
\end{rem}

\subsection{Some notation}

For $X\subset \Pi\times\M$ with closed $t$-sections we will use the notation 
$$
m(t,p)=\{(t,x)\in \Pi\times\M\mid (t,x)\in X\colon d(p,x)=\mathrm{dist}(p,X_t)\}
$$
and $m_t(p)=m(t,p)$. Then the medial axes $M_{X_t}$ correspond to the sections $M_t$ of the set 
$$M:=\{(t,p)\in\Pi\times\M\mid \exists \gamma,\tilde{\gamma}\in\Omega_{X_t,p}: \gamma\neq \tilde{\gamma}\}.$$
For a fixed $t$, $m_t$ is univalent apart from the set $M_t$.

Let us note the following general result (compare the statements about graphical convergence of sequences in \cite{RW}):
\begin{lem}\label{geral}
Assume that $Y\subset\Pi\times\M$ has closed $t$-sections, $0\in\overline{\pi(Y)\setminus\{0\}}$ for $\pi(t,x)=t$ and $Y_t\stackrel{K}{\rightarrow} Y_0$ with $Y_0\neq\varnothing$. Then for $m_t(p)=\{x\in Y_t\mid d(p,x)=d(p,Y_t)\}$, $\limsup_{\pi(Y)\ni t\to 0} m_t(p)\subset m_0(p)$, $p\in \M$. Moreover if $g=\lim \gamma_{Y_t,p}$ is a geodesic then it realises the distance $dist(p,Y_0)$. 
\end{lem}
\begin{proof}
By Remark \ref{granicecz}, given $x\in\limsup m_t(p)$, there are sequences $t_\nu\to 0$, $t_\nu\neq 0$, $m_{t_\nu}(p)\ni x_\nu\to x$. However, as $Y_0=\liminf Y_t$, we get $x\in Y_0$ and $d(x_\nu,p)=d(a,Y_{t_\nu})\to d(a,Y_0)$, by Lemma \ref{ciqglosc}, whence $x\in m_0(p)$. 

For the second part observe firstly that since $g=\lim \gamma_{Y_t,p}$ we obviously have $g(0)=x\in Y_0$ and $g(dist(a,Y_0))=a$. Now since the distance $dist(a,Y_0)$ agrees with the length of $g$, $g$ is distance realising. 
\end{proof}

\section{Introductory examples}\label{3}

In this section we restrict ourselves to the simplest and regular case of the Euclidean space and semi-algebraic sets.

Consider a closed, nonempty set $X\subset{\Rz}^k\times{\Rz}^n$ in the variables $(t,x)$. We always assume that $t=0$ is an accumulation point of $\pi(X)$.

If we put $\delta_t(x)=\mathrm{dist}(x,X_t)^2$, then we know that  \begin{enumerate}
\item $\delta_t$ is locally Lipschitz, vanishing exactly on $X_t$;
\item the non-differentiability points of $\delta_t$ coincide with $M_t$ (see \cite{Y}, \cite{BD}, for Riemannian case see \cite{A}, cite{TS});
\item $(t,x)\mapsto \delta_t(x)$ is continuous provided the sections $X_t$ vary continuously (see \cite{Dmfct} and Lemma \ref{ciqglosc}).
\end{enumerate}

Of course, there is no relation whatsoever between the medial axis of $X$ and the medial axes of the sections. However the graph $\Gamma_{m_t}$ of $m_t$ is still the $t$-section of $\Gamma_m$. 

Since we assume here $X$ to be closed, we have $\limsup X_t\subset X_{0}$ (see \cite{DD}). Note that $M_{0}$ may be empty while $\limsup M_t$ is not:
\begin{ex}
Let $X=\{(t,x,y)\in{\Rz}^3\mid x^2+y^2=t^2\}$. Then the circles $X_t\stackrel{K}{\longrightarrow} X_0=\{(0,0)\}$, but $M_t=\{(0,0)\}$ for $t\neq 0$, whereas $M_0=\varnothing$. 
\end{ex}

Note that in the Example above the dimension of $X_t$ is not preserved at the limit. Let us have a look at an Example of constant dimension:
\begin{ex}
Consider $X=\{(t,x,y)\in{\Rz}^3\mid t^2y=x^2-1\}$. Then for $t\neq 0$ each $X_t$ is a parabola $y_t(x)=(1/t^2)(x^2-1)$, whereas $X_0=\{-1,1\}\times{\Rz}$. Clearly, $X_t\stackrel{K}{\longrightarrow} X_0$. 

It is easy to see that for any parabola $X_t$, $M_t=\{0\}\times(f_t,+\infty)$ (\footnote{Note \textit{en passant} that these sets are not closed in this example.}) where $(0,f_t)=(0,y_t(0))+\frac{1}{\kappa(0,y_t(0))}\nu(0,y_t(0))$ is the focal point of $X_t$, i.e. $\kappa$ denotes the curvature and $\nu$ the unit normal which at $(0,y_t(0))$ is just $(0,1)$. We compute the curvature as $$\kappa_t(0)=\frac{y_t''(0)}{(1+y_t'(0)^2)^{3/2}}=\frac{2}{t^2}$$ and so $f_t=\frac{t^2}{2}-\frac{1}{t^2}\to -\infty$ as $t\to 0$, i.e. (\footnote{We use the following known fact for plane curves, see \cite{Th}: if $\Gamma\subset{\Rz}^2$ is a $\mathcal{C}^2$-smooth plane curve, and $x$ lies on the normal to $\Gamma$ at $a$, then $a$ is the unique point realizing $d(x,\Gamma)$ iff the segment $[a,x]$ does not contain focal points; in particular, if $||x-a||<1/\kappa(a)$ where $\kappa(a)$ is the curvature of $\Gamma$ at $a$.}) $M_t\stackrel{K}{\longrightarrow} \{0\}\times {\Rz}=M_0$.
\end{ex}

Unfortunately, the constancy of the dimension does not always guarantee the continuity of the medial axes:
\begin{ex}
Let $X=\{(t,x,y)\in{\Rz}^3\mid t^2y=x^2\}\setminus\{(0,0,y)\mid y<0\}$. It is a closed semi-algebraic set with continuously varying $t$-sections: $X_t$ is the parabola $y_t(x)=(1/t^2)x^2$ for $t\neq 0$ and the semi-line $\{0\}\times [0,+\infty)$ for $t=0$. For each parabola we have $\kappa(0)=2/t^2$ and so $$M_t=\{0\}\times(t^2/2,+\infty)\stackrel{K}{\longrightarrow}\{0\}\times[0,+\infty)\supsetneq M_0=\varnothing.$$
\end{ex}
\begin{rem}\label{glowny przyklad}
In the last example $X_0$ has a singularity, whereas the nearby fibres $X_t$ are smooth. It is important to observe that even in the definable setting, the smoothness of the limit does not necessarily imply the smoothness of the nearby sections (\footnote{In this case we also have $M_t$ a constant family for $t\neq 0$, but $M_0=\varnothing$.}):
$$
X_t=\{(x,y)\in{\Rz}^2\mid y=t|x|\}\stackrel{K}{\longrightarrow} \{(x,0)\mid x\in{\Rz}\}=X_0,\quad (t\to 0).
$$
Neither does the smoothness of the sections $X_t$, together with connectedness and constancy of dimension guarantee the smoothness of the limit:
\begin{align*}
X_t &=\{(x,t)\in{\Rz}^2\mid xy=t^2, x,y\geq 0\}\stackrel{K}{\longrightarrow}\\
\stackrel{K}{\longrightarrow}& \{(x,y)\in{\Rz}^2\mid xy=0, x,y\geq 0\}=X_0,\quad (t\to 0).
\end{align*}

These two examples are particularly interesting, since we have in both cases the best possible (from a set-theoretic point of view) situation: the sets in question are graphs converging locally uniformly (\footnote{In particular, the convergence is without multiplicities, i.e. to each branch corresponds one branch in the converging sets. Note also that $v_\nu(u)=\sqrt{u^2+(1/\nu)}$  is an example of a sequence of 1-Lipschitz $\mathcal{C}^1$ functions converging locally uniformly to a non-differentiable function.}). 

Indeed, in the second case, after the obvious change of variables $u=(x-y)/2$, $v=(x+y)/2$, we have $X_t$ described by $v^2-u^2=t^2$ in $\{(u,v)\mid v\geq |u|\}$. Thus, $X_t$ is the (analytic) graph of the function $v(u)=\sqrt{u^2+t^2}$. From its symmetry we infer, computing as earlier the curvature at the origin, that $M_t=\{0\}\times (2|t|,+\infty)$. So that $M_t\stackrel{K}{\longrightarrow} \{0\}\times [0,+\infty)=\overline{M_0}$.
\end{rem}

These examples lead to the following first natural conjecture. 

\medskip
\noindent\textbf{Conjecture 1.} %
{\it Assume that there is $X_0=\limsup X_t$ for the definable set $X$, 
then $\limsup M_t\supset M_{0}$.
}


\smallskip
However, it turns out immediately that Conjecture 1 is false:
\begin{ex}\label{3.5}
Consider the closed, semi-algebraic set $X\subset{\Rz}\times{\Rz}^2$ defined by the sections $X_0:=\{(x,y)\in{\Rz}^2\mid y=|x|\}$ and $X_t=\{(x,y)\in{\Rz}^2\mid y=\mathrm{sgn}(t)x, \mathrm{sgn}(t)x\geq 0\}$ for $t\neq 0$. Then $X_0=\limsup X_t$, but $\liminf X_t=\{(0,0)\}$ so that there is no convergence. We have  $M_0=\{0\}\times(0,+\infty)$, whereas for $t\neq 0$, $M_t=\varnothing$. 
\end{ex}

We modify the conjecture:

\smallskip
\noindent\textbf{Conjecture 2.} {\it If $X_t$ have a limit at $0$ (not necessarily coinciding with $X_0$, but at least included in it), then there exists also $\lim M_t$ and it contains $M_{0}$.}

\smallskip
\begin{rem}
Note (cf. Example \ref{cg} hereafter) that the convergence of $M_t$ is to be considered over $\pi(M)$.
\end{rem}

If $X_t$ does not converge, then in general neither $M_t$ does:
\begin{ex}
Let $X\subset{\Rz}\times{\Rz}^2$ be the closed semi-algebraic set defined by $X_0=\{(x,y)\in{\Rz}^2\mid x^2=y^2\}$ and for $t\neq 0$, $X_t=\{(x,y)\in{\Rz}^2\mid y=\mathrm{sgn}(t)|x|\}$. Then $X_0=\limsup X_t$, $\liminf X_t=\{(0,0)\}$ so that there is no convergence. Now, $M_t=\{0\}\times (0,+\infty)$  or $\{0\}\times (-\infty,0)$ according to the sign of $t\neq 0$. Therefore, 
\begin{align*}
\liminf M_t &=\{(0,0)\}\subsetneq \limsup M_t=\\ 
&=\{0\}\times{\Rz}\subsetneq \overline{M_0}=\{(x,y)\in{\Rz}^2\mid xy=0\}.
\end{align*}
\end{ex}

Observe that by the results of \cite{DD} we may assume that $X$ has continuously varying sections (we lose, however, the closedness of $X$, since this assumption requires getting rid of a nowheredense subset of $\pi(X)$). 
\begin{ex}\label{cg}
Note that the assumptions:\\ \centerline{$X$ closed and definable, $X_0=\lim X_t$,} \\ do not imply necessarily that the nearby sections are continuous. To see this consider two examples. 

The first one is the subanalytic set $X=({\Rz}\times\{0\})\cup\bigcup_{n=1}^{+\infty}\{(1/n,n)\}$. Of course this is not definable. 

The second one is a general semi-algebraic example but with two-dimensional parameters: $X=({\Rz}^2\times\{0\})\cup\{(x,0,x)\mid x\in{\Rz}\}$.  By \cite{DD}, apart from a nowheredense set in the parameters, all the sections are continuous. Note that in this example we have to throw away exactly those parameters over which the $M_t$'s are non-void.
\end{ex}

\begin{ex}
Consider $X$ given by 
$$
X_t=\begin{cases}
\mathbb{S}^1\cup\{(0,0)\}, &t>0;\\
\mathbb{S}^1\cup (1/2)\mathbb{S}^1, &t<0;\\
\mathbb{S}^1\cup (1/2)\mathbb{S}^1\cup\{(0,0)\}, &t=0.
\end{cases}
$$
This is a closed semi-algebraic set with $X_0=\limsup X_t$ and $\liminf X_t=\mathbb{S}^1$. Here 
$$
M_t=\begin{cases}
(1/2)\mathbb{S}^1, &t>0;\\
(3/4)\mathbb{S}^1\cup\{(0,0)\}, &t<0;\\
(3/4)\mathbb{S}^1\cup(1/4)\mathbb{S}^1, &t=0
\end{cases}
$$
and so $\liminf M_t=\varnothing$, $\limsup M_t=(1/2)\mathbb{S}^1\cup(3/4)\mathbb{S}^1\cup\{(0,0)\}$ which shows that there is no relation whatsoever with $M_0$.
\end{ex}

One final Example to show that in Conjecture 2 the limit $\lim M_t$ depends on how the sets $X_t$ converge rather than on the limit $X_0$:
\begin{ex}
Let $X_0={\Rz}\times\{0\}$, and consider $$X_1=((-\infty,-1]\times\{0\})\cup\{(x,-\mathrm{sgn}(x)x+1)\mid |x|\leq 1\}\cup ([1,+\infty)\times\{0\}).$$
The two simplest ways of making $X_1$ continuously evolve to $X_0$ are the following: we define for $t\in (0,1)$, the sets $X_t$ either as
$$
((-\infty,-(1/t)]\times\{0\})\cup\{(x,-t^2\mathrm{sgn}(x)x+t)\mid |x|\leq 1\}\cup ([(1/t),+\infty)\times\{0\})
$$
or as
$$
((-\infty,-1]\times\{0\})\cup\{(x,-t\mathrm{sgn}(x)x+t)\mid |x|\leq 1\}\cup ([1,+\infty)\times\{0\}).
$$
In both cases the sets $M_t$ converge and the limit contains $\{0\}\times(-\infty,0]$ (actually, it reduces to it in the first case), but in the second case it contains also $\{-1,1\}\times [0,+\infty)$.
\end{ex}
Now we give a semi-algebraic counter-example to Conjecture 2:
\begin{ex}\label{3.11}
Consider the set $X=\{(t,x,y)\in{\Rz}\times{\Rz}^2\mid y=t|x|\}$ from Remark \ref{glowny przyklad}. It is semi-algebraic, we have $X_t\stackrel{K}{\longrightarrow} X_0$, but $$M_t=\begin{cases}
\{(x,y)\mid x=0, y>0\}, &t>0,\\
\varnothing, & t=0,\\
\{(x,y)\mid x=0, y<0\}, &t<0,
\end{cases}
$$
so that there is no convergence.
\end{ex}
Observe that in all the previous examples with converging sections, we had $\liminf M_t\supset M_0$. This remark leads to the main theorem presented in the next section.

\section{Inner semi-continuity of medial axes}
\begin{thm}\label{main}
Assume that $X\subset \Pi\times\M$ has closed $t$-sections and we have the convergence $X_t\stackrel{K}{\longrightarrow} X_0$. Then for $M=\{(t,x)\in\Pi\times \M\mid \exists \gamma_{X_t,p},\tilde{\gamma}_{X_t,p}\in\Omega_{X_t,p}: \gamma_{X_t,p}\neq \tilde{\gamma}_{X_t,p}\}$, we have $$\liminf_{\pi(M)\ni t\to 0} M_t\supset  M_0$$
where we posit $\liminf M_t=\varnothing$ when $0\notin \overline{\pi(M)\setminus\{0\}}$.
\end{thm}
\begin{rem}
The Theorem implies that $0$ cannot be an isolated point of $\pi(M)=\{t\mid M_t\neq\varnothing\}$, i.e. $M_0=\varnothing$, if $0\notin\overline{\pi(M)\setminus\{0\}}$.
\end{rem}
\begin{rem}\label{Rk}
Example \ref{3.11} shows that we can hardly expect a better result even in the quite regular situation when we are dealing with a convergent semi-algebraic one-parameter family of graphs.
\end{rem}
\begin{rem}
The result presented in Theorem \ref{main} is what is called {\it inner semi-continuity} at the origin of the multifunction $t\mapsto M_t$.
\end{rem}
\begin{ex}\label{bd}
Before we prove this Theorem, let us give one important application used in \cite{BD}. In the semi-algebraic (or broader  -- definable ) setting, due to the Curve Selecting Lemma, the tangent cone $C_0(X)$ to a semi-algebraic set $X\subset{\Rz}^n$ at $x=0$ is obtained as the limit
$$
C_0(X)=\lim_{t\to 0^+}(1/t) X.
$$
In particular, if we know that $C_0(X)$ has a nonempty medial axis, then we conclude using the Theorem above, not only that $M_X\neq\varnothing$ but also that $0\in\overline{M_X}$. Moreover, since it is obvious that $M_{(1/t)X}=(1/t)M_X$ and we are semi-algebraic, we obtain even the convergence of the medial axes of the dilatations and the limit is clearly $C_0(M_X)$. The Theorem above gives then $C_0(M_X)\supset M_{C_0(X)}$.
\end{ex}

\begin{proof}[Proof of Theorem \ref{main}] 
Assume that there is a point $$a\in M_0\setminus\liminf M_t.$$ There are two cases to deal with depending on whether $t=0$ is an accumulation point of $\pi(M)$, or not. We will treat them simultaneously. 

If $0\in\overline{\pi(M)\setminus\{0\}}$, this implies that for some ball $B$ centred at $a$ and for any neighbourhood $V\ni 0$, there is a point $t\in \pi(M)\cap V\setminus\{0\}$ such that $M_t\cap B=\varnothing$. In other words $t=0$ belongs to the closure of the set 
$$
E:=\{t\in\pi(M)\setminus\{0\}\mid M_t\cap B=\varnothing\}.
$$
This set with topology induced by $\Pi$ inherits the topological properties we are intrested in, namely the $T_1$ class and countable base of $0$. Hence, we can shrink the parameter space to $E$ conserving the compliance with the assumptions of the theorem. Then the set $$X':=\{(\tau,x)\in E\times \M\mid (\tau,x)\in X\}$$ has closed $t-$sections and satisfies $X'_\tau=X_{\tau}\stackrel{K}{\longrightarrow} X'_0=X_0$, when $\tau\to 0$. 

On the other hand, if $0$ is isolated in $\pi(M)$, then we may take $X'$ to be simply $X\cap 
(V\times\M)$  with $V\cap \pi(M)=\{0\}$ and we still have the convergence of $X'_\tau$ to $X_0$ together with a neighbourhood of $a$ in which there are no medial axes for parameters $t\neq 0$.

This means that we may restrict our considerations to a family of converging sets $X_t\stackrel{K}{\longrightarrow} X_0$ ($t\to 0$) accompanied by a neighbourhood $U\ni a$ such that all the $m_t$'s are univalent in $U$ for $t\neq 0$ and every point of $U$ admits exactly one distance minimising geodesic to $X$. Put $d_t(x):=\mathrm{dist}(x,X_t)$. Now, define
$$
r(t):=\sup\{s\geq 0\mid \mathbb{B}(\gamma_{X_t,a}(s),s)\cap X_t=\varnothing\},\> t\in T\backslash\{0\},
$$
i.e. we consider the open ball $\mathbb{B}(a,d_t(a))$ and start to inflate it from the point $m_t(a)$ along geodesic $\gamma_{X_t,a}$ while keeping the `tangency' point $m_t(a)$ to $X_t$ --- we do this as long as possible without meeting $X_t$. Naturally, there is $r(t)\geq d_t(a)$ for all $t$ as the open ball $\mathbb{B}(a,d_t(a))=\mathbb{B}(\gamma_{X_t,a}(d_t(a)),d_t(a))$ is always disjoint with $X_t$.

First, let us observe that by shrinking the parameter space $\Pi$ once more we can assume the geodesics segments $\gamma_{X_t,a}$ to converge to a single geodesic $\gamma_{X_0,a}$ and the function $\mu(t):=\gamma_{X_t,a}(0)$ to be continuous at the origin. Indeed the geodesics $\gamma_{X_t,a}$ are of form $\gamma_{X_t,a}(s)=exp_a((1-s)\cdot v_t)$ for an appropriate $v_t\in T_a\M$. Since $1/2<\|v\|< 2$ for $t$ close to $0$, it is easy to find a convergent sequence $v_{t_\nu}\to v_0$. It suffices then to reduce $\Pi$ again to the newly found sequence. 

Once we have established this, we observe that $a\in M_0$ implies that
$$
\sup\{s\geq 0\mid \mathbb{B}(\gamma_{y,a}(s),s)\cap X_0=\varnothing\}=d_0(a)
$$
where $y$ is the endpoint of the geodesic $\gamma_{X_0,a}$ defined earlier as the limit of $\gamma_{X_t,a}$. For, should there be $r(0)>d_0(a)+c$ for some $c>0$, then we would have  $d_0(\gamma_{y,a}(d_0(a)+c))=d(a)+c$. But at the same time, taking any geodesic $h:=\tilde{\gamma}_{X_0,a}\neq \gamma_{y,a}=:g$ and $\delta$ small enough we would get $$d_0(g(d_0(a)+c)< d(h(0),h(d_0(a)-\delta))+d(h(d_0(a)-\delta),g(d_0(a)+\delta))+$$
$$d(g(d_0(a)+\delta),g(d_0(a)+c)<d(a)+c$$

If we succeed in showing that $\lim r(t)=d_0(a)$, we are done, as it means that $\gamma_{X_t,a}(r(t))$ belongs to $Cut(X_t)$, for $t$ close to zero, where $Cut(X_t)$ denotes the cut locus of $X_t$ 
. Indeed, the points $\gamma_{X_t,a}(r(t))$ certainly belong to $Cut(X_t)$ and we know that 
$M_t\subset Cut(X_t)\subset\overline{M_t}$ \cite{A} (and the closures do not alter the limit). This gives the desired contradiction.

Why is $\lim r(t)=d_0(a)$? Observe that by definition $r(t)\geq d_t(a)$ and the function $t\to d_t(a)$ is continuous, whence $\liminf r(t)\geq d_0(a)$. Suppose that $\limsup r(t)>1+c$, for a certain $c>0$. Then by setting $$B_t:=\mathbb{B}(\gamma_{X_t,a}(d_0(a)+c),d_{t}(\gamma_{X_t,a}(d_0(a)+c)))$$ we have $B_t \cap X_{t}=\varnothing$. Thanks to Lemma \ref{ciqglosc} these balls converge to the closure of a ball $$B_0:=\mathbb{B}(\gamma_{y,a}(d_0(a)+c),d_0(\gamma_{y,a}(d_0(a)+c)).$$ 

But the latter open ball can contain no point of $X_0$, otherwise such
a point would be reachable by points from the sets $X_{t}$, due to the
convergence, which is clearly impossible. Namely, if $x_0 \in B_0 \cap X_0$, then
taking a ball $U$ centred at $x_0$ and such that $U \subset B_0$ we would get
\begin{align*}
&\partial B_t \xrightarrow{K} \partial B_0\text{ and }\partial B_0 \cap \partial U = \varnothing \Rightarrow \partial B_t \cap \partial U = \varnothing \text{, }t \text{ close to } 0,\\
& \overline{B_t} \xrightarrow{K} \overline{B_0}\text{ and }\overline{B_0} \cap U \neq \varnothing \Rightarrow \overline{B_t} \cap U \neq \varnothing \text{, }t \text{ close to } 0,
\end{align*}
whence $U \subset B_t$ for all $t$ close to $0$. But $X_{t} \to X_0$ and $x_0 \in U \cap X_0$ implies that $X_{t}  \cap U \neq \varnothing$ and so $X_{t}$ meets $B_\nu$ for all $t$ close to $0$
which contradicts the choice of $B_t$.
Therefore, $B_0 \cap X_0 = \varnothing$ which coupled with $int B_0 \supsetneq \mathbb{B}(a, d_0(a))$ contradicts $a \in M_0$. The proof is accomplished.

\end{proof}

\section{Inner semi-continuity of conflict sets in continuous families}
In this section we would like to address the question of the behaviour of the conflict set of a parametrised family of sets. Once again it is convenient to interpret the situation in terms of the sections of a set $X\subset \Pi\times \M$ with parameters $t$ in a $T_1$ topological space $\Pi$ with a distinguished non-isolated point $0$ admitting a countable base.

Before going any further let us recall some basic notions (compare \cite{BS}).
\begin{dfn}\label{cset}
We define the {\it conflict set} of a family $\mathcal{X}:=\{X_1,\dots, X_k\}$ of $k\geq 2$ closed, pairwise disjoint, nonempty subsets of $\M$ to be 
$$
\mathrm{Conf}(\mathcal{X})=\{p\in\M \mid \exists i\neq j\colon d(p,X_i)=d(p,X_j)\leq \varrho(p)\}
$$
where and $\varrho(p):=\min_{i=1}^k d(p,X_i)=d(p,\bigcup_{i=1}^k X_i)$. 
It is useful to introduce also the notion of the {\it open territory of $X_j$ (with respect to the family $\mathcal{X}$)} --- it is the open set
$$
\mathrm{Terr}(X_j,\mathcal{X}):=\{p\in\M\mid d(p,X_j)<\min_{i\neq j}d(p,X_i)\}
$$
whose closure is equal to $\overline{\mathrm{Terr}}(X_j,\mathcal{X}):=\{p\in\M\mid d(p,X_j)\leq \varrho(p)\}$. 
We call the latter the {\it closed territory of} $X_j$. 
\end{dfn}
If we are dealing with two sets, i.e. $k=2$, then their conflict sets consists of points that are equidistant to both of them. Thus a parabola is just the conflict set of a line and a point outside it in the Euclidean plane.
\begin{rem}\label{suma}
Most often we shall use the notation ${T}^j={\mathrm{Terr}}(X_j,\mathcal{X})$ and $\overline{T}^j=\overline{\mathrm{Terr}}(X_j,\mathcal{X})$, while endeavouring to avoid any confusion. It is easy to see that
$$
\mathrm{Conf}(\mathcal{X})=\bigcup_{1\leq i<j\leq k} (\overline{T}^i\cap \overline{T}^j).
$$
In particular, if the sets $X_j$ are definable in an o-minimal strucutre, then so is their conflict set. Moreover, we also have 
$$
\M\setminus\mathrm{Conf}(\mathcal{X})=\bigcup_{i=1}^k T^j
$$
and the union on the right-hand side is disjoint.
\end{rem}

The relation of the conflict set of a family to the medial axis of its union is worth noting:
\begin{lem}
If $\mathcal{X}=\{X_1,\dots, X_k\}$ is a family of closed, pairwise disjoint, nonempty subsets of $\M$ with $k\geq 2$, then 
$$
M_{\bigcup_{i=1}^k X_i}=\mathrm{Conf}(\mathcal{X})\cup\bigcup_{j=1}^k[\mathrm{Terr}(X_j,\mathcal{X})\cap M_{X_j}]
$$
all the unions appearing here being disjoint ones.
\end{lem}
\begin{proof}
Obviously $d(p,\bigcup X_i)=\varrho(p)$ which implies the inclusion `$\supset$'. On the other hand, if we pick a point $a\in M_{\bigcup X_i}$, then there are at least two distinct geodesics $g,h$ with endpoints in $\bigcup X_i$ realising the distance $d(a,\bigcup X_i)$. If their endpoints belong to two different sets, say $g(0)\in X_1, h(0)\in X_2$, then $a\in \mathrm{Conf}(\mathcal{X})$. Otherwise, if all the endpoints of minimising geodesics lie in the same set: $X_j$, then $\varrho(a)=d(a,X_j)<\min_{i\neq j} d(a,X_i)$ so that $a\in \mathrm{Terr}(X_j,\mathcal{X})$ and obviously $a\in M_{X_j}$.
\end{proof}
\begin{rem}
In the case when each of the sets $X_i=\{x_i\}$ from Definition \ref{cset} is a singleton, the conflict set is called a {\it Voronoi diagram} and it coincides with the medial axis the union $\bigcup_{i=1}^k X_i=\{x_1,\dots,x_k\}$. Therefore, in this particular case (which, however, is very important in view of all the applications of Voronoi diagrams, let it be in computer vision, economy, geography and so on) we already have the inner semi-continuity result from Theorem \ref{main} (compare \cite{Re}, \cite{Ro}).
\end{rem}

\subsection{The setting}\label{konflikt}
Our present aim is to study the behaviour of the conflict sets of a family of sets $\mathcal{X}_t$ with a multidimensional parameter $t$. Let us make this more precise. We consider $k\geq 2$ closed, pairwise disjoint sets $X^1,\dots,X^k\subset \Pi\times\M$ with the following premises:\begin{itemize}
\item  $\pi(X^1)=\ldots=\pi(X^k)$ where $\pi(t,p)=t$; 
\item  $0\in \overline{\pi(X^i)}$ but $0\notin \pi(X^i)$;
\item for each $i\in \{1,\dots,k\}$ there exists a Kuratowski limit $X_0^i:=\lim_{t\to 0} X_t^i$ (maybe empty).
\end{itemize}

Put $X:=\bigcup_{i=1}^k X^i$; it is a disjoint union. Also, $X_t\stackrel{K}{\rightarrow} \bigcup_{i=1}^k X_0^i=:X_0$. 
Write also $\tilde{\Pi}:=\pi(X^i)$ for the common projection.

\begin{rem}
In the setting introduced above, $\tilde{X}:=X\cup (\{0\}\times X_0)$ is a set with closed $t-$sections and $\tilde{X}_t\stackrel{K}{\rightarrow} \tilde{X}_0$. 

This raises the question why not to start our considerations from a set $X\subset \Pi\times \M$ with closed $t$-sections that are Kuratowski-continuous at $t=0$ and such that for any $t\in \pi(X)\setminus\{0\}$ we are given a decomposition (\footnote{I.e. we present $X_t$ as a disjoint union of nonempty sets and call $\mathcal{X}_t$ their family.}) $\mathcal{X}_t$ of $X_t$ into $k\geq 2$ (independent of $t$) pairwise disjoint, nonempty sets. Is it possible then to renumber the sets of each family $\mathcal{X}_t$ in such a way that it would yield a decomposition   of $X\setminus(\{0\}\times X_0)$ into $k$ pairwise disjoint, nonempty, definable sets $X^i$ with a common projection $\Pi$ and each having convergent $t$-sections when $t\to 0$? The following example shows this may not be possible even in the semi-algebraic case, which explains why we adopt a seemingly more na\"\i ve approach.
\end{rem}
\begin{ex}\label{AB} 

Consider a ZZ-shaped semi-algebraic set  $$X:=([-1,1]\times\{0,1\})\cup \{(x,x)\mid x\in [0,1]\}\cup \{(x-1,x)\mid x\in [0,1]\}\subset{\Rz}^2$$ together with a global decomposition $X=X^1\cup X^2$ into two disjoint semi-algebraic sets sharing the same projection $(-1,1)$, then even though $X_t\stackrel{K}{\rightarrow} X_0$, the sets $X^i$ may not have convergent sections at the origin. For instance, if $X^1=(-1,1)\times\{0\}$ and $X^2$ is the closure of $X\setminus X^1$ in $(-1,1)\times{\Rz}$, then $X^1_t\stackrel{K}{\rightarrow}X_0^1$ but $X^2_t$ does not converge at the origin. On the other hand, if $\tilde{X}^1=((-1,0]\times\{0\})\cup ([0,1)\times\{1\})\cup\{(x,x)\mid x\in [0,1)\}$ and $\tilde{X}^2$ is the closure $X\setminus \tilde{X}^1$ in $(-1,1)\times{\Rz}$, then neither of the sets has convergent sections at the origin.
\end{ex}


Let us come back to our general situation. Denote by $\tilde{\mathcal{X}}^1,\dots, \tilde{\mathcal{X}}^{k_0}$ the equivalence classes of the equivalence relation defined on  $\mathcal{X}:=\{X^1,\dots,X^k\}$ in the following way: we identify $X^i\sim X^j$, whenever either $X_0^i=X_0^j=\varnothing$, or  there are $r\geq1$ indices $i_0, i_1,\dots,i_r\in\{1,\dots,k\}$ with $i_0=i,i_r=j$ and such that $X_0^{i_s}\cap X_0^{i_{s+1}}\neq\varnothing$, for $s=0,\dots,r-1$. Then $1\leq k_0\leq k$ and we may assume that the original sets $X^i$ are numbered in such a way that the set of indices $\{1,\dots,k\}$ is divided into $k_0$ sets of the form $\{l_{j-1}+1,l_{j-1}+2,\dots, l_j\}$, for $j=1,\dots,k_0$, where $l_0=0<l_1=1<\ldots<l_{k_0}=k$, and $\tilde{\mathcal{X}}^j=\{X^{l_{j-1}+1},\dots, X^{l_j}\}$. The following notion will be useful.

\begin{dfn}
Given a family $\mathcal{Y}=\{Y^1,\dots, Y^k\}$ of $k$ pairwise disjoint, nonempty sets, we will say that the family $\tilde{\mathcal{Y}}=\{\tilde{Y}^1,\dots,\tilde{Y}^{k_0}\}$ consisting of pairwise disjoint, nonempty sets is a {\it regrouping} of $\mathcal{Y}$, if each of the sets $\tilde{Y}^j$ is the union of some of the sets $Y^i$ (\footnote{Then $k_0\leq k$.}).
\end{dfn}

Now put $\tilde{X}^j:=\bigcup\tilde{\mathcal{X}}^j$, $j=1,\dots, k_0$. This induces a decomposition $\mathcal{X}_0=\{\tilde{X}_0^1,\dots,\tilde{X}_0^{k_0}\}$ of $X_0$ satisfying $\tilde{X}_t^j\stackrel{K}{\rightarrow}\tilde{X}_0^j$, for $j=1,\dots, k_0$, and each $\tilde{X}_0^j$ is the union of the sets $X_0^i$ corresponding to the indices $i$ attached to $\tilde{\mathcal{X}}^j$. Then $\tilde{\mathcal{X}}_t:=\{\tilde{X}_t^1,\dots,\tilde{X}_t^{k_0}\}$ is a regrouping of the family $\mathcal{X}_t:=\{X_t^1,\dots,X_t^k\}$.

Let us look at a simple illustration of what happens here:
\begin{ex}\label{przyklad niezwarty}
Consider in ${\Rz}\times{\Rz}$ the semi-algebraic sets 
\begin{align*} &X^1=\{(x,1/x)\mid x\in (0,1)\},\\ &X^2=\{(x,y)\mid x\in (0,1), x/2\leq y\leq 1/2\},\\ &X^3=(0,1)\times\{3/4\},\\ &X^4=[(0,1)\times[-1/2,0]]\cup\{(x,-1/x)\mid x\in (0,1)\},\\ &X^5=\{(x,y)\mid x\in (0,1), -1\leq y\leq -(x+1)/2\}.\end{align*} 
Then $X_0^1=\varnothing$, $X_0^2=[0,1/2]$, $X_0^3=\{3/4\}$, $X^4=[-1/2,0]$, and of course $X_0^5=[-1,-1/2]$ so that we obtain three equivalence classes: $\tilde{\mathcal{X}}^1=\{X^1\}$, $\tilde{\mathcal{X}}^2=\{X^2,X^4,X^5\}$, $\tilde{\mathcal{X}^3}=\{X^3\}$. It is obvious that we cannot avoid merging $X^2$ with $X^4$ and $X^5$. The corresponding decomposition $\mathcal{X}_0$ is given by the family $\{\varnothing,[-1,1/2], \{3/4\}\}$. Of course, in order to properly compute $\mathrm{Conf}(\mathcal{X}_0)$ we have first to discard the emptyset from the family.
\end{ex}

\begin{rem}\label{5.9}
It is pretty obvious that in order to study the limit behaviour of the conflict sets, we need first to define well the limit family of sets whose conflict set we will then have to determine. Since we allow the limit sets $X_0^i$ to overlap, we have to regroup them. Any regrouping will naturally affect the converging families in that they too will need to be regrouped in a new manner, while keeping track of what happens then with their conflict sets. At the same time we want to keep the setting we are in (the continuity and the common projection). This aim can be achieved, of course, in many ways. However, the main idea is that we would like to obtain the maximal possible number of sets in the regrouping of the limit sets $X_0^i$, i.e. we cannot avoid the merging of sets whose limits overlap. 
And that is precisely what is ensured by the construction described above: $k_0$ is the maximal cardinality possible and the regrouping $\tilde{\mathcal{X}}$ (i.e. each $\tilde{\mathcal{X}}_t$) is clearly a canonical one.
\end{rem}

Denote by $C_t$ the conflict set $\mathrm{Conf}(\mathcal{X}_t)$ and $\tilde{C}_t:=\mathrm{Conf}(\tilde{\mathcal{X}}_t)$, for $t\in \Pi$ and put $C_0=\mathrm{Conf}(\tilde{\mathcal{X}}_0)$ where $\tilde{\mathcal{X}}_0:= \mathcal{X}_0\setminus\{\varnothing\}$ (note that one of the sets $\tilde{X}_0^j$ may be empty) with the convention that $C_0=\varnothing$, when $\#\tilde{\mathcal{X}}_0\leq 1$. 


\begin{ex}
In Example \ref{przyklad niezwarty} we obtain $C_0=\{5/8\}$ and $\tilde{C}=\{(x,y)\mid x\in (0,1), y=1/(2x)+3/8\}\cup[(0,1)\times C_0]$.
\end{ex}

There is a nice relation between the conflict set of a family and the conflict set of a regrouping of this family.
\begin{lem}\label{zawieranie}
Let $\tilde{Y}=\{\tilde{Y}^1,\dots,\tilde{Y}^{k_0}\}$ be a regrouping of a family $\tilde{Y}=\{Y^1,\dots,Y^k\}$ of closed, nonempty subsets of $\M$. Then $\mathrm{Conf}(\tilde{\mathcal{Y}})\subset\mathrm{Conf}(\mathcal{Y})$.
\end{lem}
\begin{proof}
For $s\in\{1,\dots, k_0\}$, write $\tilde{Y}^s=\bigcup_{i=1}^{r_s}Y^{s_i}$, so that $k=r_1+\ldots+r_{k_0}$. Take $x\in \mathrm{Conf}(\tilde{\mathcal{Y}})$. By Remark \ref{suma} his is equivalent to say that there are indices $i\neq j$ such that
$$
\begin{cases}
d(x, \tilde{Y}^i)\leq \min_{s=1}^{k_0}d(x,\tilde{Y}^s),\\
d(x, \tilde{Y}^j)\leq \min_{s=1}^{k_0}d(x,\tilde{Y}^s)=\min_{s=1}^{k_0}\min_{i=1}^{r_s}d(x,{Y}^{s_i}).
\end{cases}
$$
Observe that $\min_{s=1}^{k_0}\min_{i=1}^{r_s}d(x,{Y}^{s_i})=\min_{i=1}^kd(x,Y^i)=:\varrho(x)$ and since $d(x,\tilde{Y}^j)=\min_{i=1}^{r_j}d(x,Y^{j_i})$, we conclude that there are indices $i_0\in\{i_1,\dots, i_{r_i}\}$ and $j_0\in \{j_1,\dots, j_{r_j}\}$ (necessarily distinct) satisfying 
$$
\begin{cases}
d(x, {Y}^{i_0})\leq d(x, \tilde{Y}^i)\leq \varrho(x),\\
d(x, {Y}^{j_0})\leq d(x, \tilde{Y}^j)\leq\varrho(x).
\end{cases}
$$
Therefore, $\mathrm{Conf}(\mathcal{Y})$ as required. 
\end{proof}

We are now ready to prove the main result of this section.
\begin{thm}\label{CS}
In the setting introduced above, there is $$C_0=\mathrm{Conf}(\tilde{\mathcal{X}}_0)\subset\liminf_{\tilde{\Pi}\ni t\to 0} \mathrm{Conf}({\mathcal{X}}_t)=\liminf_{\tilde{\Pi}\ni t\to 0} C_t.$$
Moreover, we have the convergence $$C_0=\lim_{\tilde{\Pi}\ni t\to 0}\mathrm{Conf}(\tilde{\mathcal{X}}_t)=\lim_{\tilde{\Pi}\ni t\to 0} \tilde{C}_t.$$
\end{thm}
\begin{proof}
The proof starts similarly to the medial axis case. We will be using the notation introduced so far. 

Suppose that there is a point $a\in C_0\setminus\liminf C_t$.  By Lemma \ref{zawieranie}, we have $a\in C_0\setminus\liminf \tilde{C}_t$. Thus, there is a ball $B$ centred at $a$ such that arbitrarily near 0 we can find $t\in \tilde{\Pi}$ such that $B\cap \tilde{C}_t=\varnothing$. By replacing the parameter space by a suitable sequence, we can assume that the intersection is empty for all $t\in \tilde{\Pi}$. In other words we can simply consider the situation where
$\tilde{\mathcal{X}}=\{\tilde{X}^1,\dots,\tilde{X}^{k_0}\}$ with $\tilde{X}^j_t\stackrel{K}{\rightarrow} \tilde{X}_0^j$, $\tilde{\Pi}\ni t\to 0$, and $B\cap\tilde{C}_t=\varnothing$, for $t\in \tilde{\Pi}$ and $j=1,\dots, k_0$. 
Of course, the limit sets are taken from the induced decomposition $\mathcal{X}_0$ as introduced above. By construction, for each $t\in\tilde{\Pi}$ there is a unique index $j(t)\in\{1,\dots,k_0\}$ such that $B\subset \mathbb{T}^{j(t)}_t$ where 
$$
\mathbb{T}^j=\{(t,p)\in \tilde{\Pi}\times\M \mid d(p,\tilde{X}_t^j)<\min_{i\neq j}d(p,\tilde{X}_t^i)\}.
$$
Indeed, if there were two points $x,y\in B$ belonging to two different open territories $\mathbb{T}^i_t$ and $\mathbb{T}_t^j$, respectively, then since these sets are disjoint and $B$ is connected, we would have 
$\varnothing\neq B\cap\overline{\mathbb{T}_t^i}\cap\overline{\mathbb{T}_t^j}\subset B\cap \tilde{C}_t$,contrary to the assumptions, due to Darboux theorem applied to any curve joining $x,y$ in $B$ and the function $p\mapsto\prod_{l\neq i} (d(p,X_i)-d(p,X_l))$. Since the number of possible indices is finite we can assume additionally that $j(t)=j_0$ for all $t\in \tilde{\Pi}$ by shrinking the parameter space again.

Now, we observe that for any $j\in\{1,\dots,k_0\}$, $$\mathbb{T}^{j}_t\stackrel{K}{\rightarrow}\{p\in\M\mid d(p,\tilde{X}_0^{j})\leq \min_{i=1}^{k_0}d(p,\tilde{X}_0^i)\}=:\overline{\mathbb{T}}^{j}_0\leqno{(\dag)}$$
where obviously $\overline{\mathbb{T}}^{j}_0$ reduces to $\varnothing$ when $\tilde{X}_0^j=\varnothing$ i.e. $d(x,\tilde{X}_0^{j})=+\infty$ (note that we necessarily have $k_0\geq 3$ so that there are at least two nonempty $\tilde{X}_0^i$, since $C_0\neq\varnothing$ by the assumptions.). 

Let us prove this convergence. Lemma \ref{ciqglosc} gives the assertion for the case $\overline{\mathbb{T}}^{j}_0=\varnothing$ and thus we may assume that we are dealing with $\tilde{X}_0^j\neq\varnothing$. If we put $\varrho(t,x)=\min_{i=1}^{k_0}d(x,\tilde{X}_t^i)$, then $\overline{\mathbb{T}^{j}_t}$ is described by the inequality $d(p,\tilde{X}^{j}_t)\leq \varrho(t,p)$ and thus $$\limsup\overline{\mathbb{T}^{j}_t}\subset \overline{\mathbb{T}}_0^{j},$$ cf. Remark \ref{granicecz}. 
On the other hand, if we pick a point $p\in \mathbb{T}^j_0$ then $d(p,\tilde X^j_0 ) < \varrho(0, p)$ and for some $\varepsilon > 0$, there is $$d(p,\tilde{X}^j_0)+\varepsilon<d(p,\tilde{X}^i_0)-\varepsilon,\,\forall i\neq j.$$ Thus for $t$ suitably close to $0$ $d(p,\tilde{X}^j_t)<d(p,\tilde{X}^i_t)$ thanks to Lemma \ref{ciqglosc} and, consequently, $p\in \mathbb{T}^j_t$.

Once we have established the convergence $(\dag)$, we are done. Indeed, on the one hand, $B\subset \mathbb{T}_t^{j_0}$ implies $B\subset\overline{\mathbb{T}}_0^{j_0}$. On other hand, since $a\in {C}_0$, then there must exist an index $i\neq j_0$ such that $a\in\overline{\mathbb{T}}_0^i$ and so $B\cap \overline{\mathbb{T}}_0^i\neq\varnothing$. But then the convergence $(\dag)$ implies that $B\cap \mathbb{T}_t^i\neq\varnothing$, for all $t$ sufficiently close to the origin, which is impossible in view of the fact that $B\subset \mathbb{T}_t^{j_0}$. This contradiction ends the proof of the first assertion. 

Now, for the `moreover' part, we already have $C_0\subset\liminf\tilde{C}_t$ by the first part of the proof and we need only to show that $\limsup\tilde{C}_t\subset C_0$. Thus, fix $a\in \limsup\tilde{C}_t$. There is $\{0\}\times \limsup\tilde{C}_t=\overline{\tilde{C}}\cap(\{0\}\times\M)$ (compare \cite{DD} Proposition 2.7) which means that $(0,p)\in\overline{\tilde{C}}\setminus\tilde{C}$, as $0\notin \pi(X^i)$. 
Therefore we can find points $t_\nu\in \Pi$ and $(t_\nu,a_\nu)\in \tilde{C}$, $\nu\in \mathbb{N}$ such that $(t_\nu,a_\nu)\to (0,a)$ as $\nu\to \infty$.
We keep the notation $\mathbb{T}^j$ introduced in the first part of the proof.

Suppose now that $a\notin C_0$. Then there is exactly one index $i\in\{1,\dots,k_0\}$, for which $a\in \mathbb{T}_0^i$ where of course
$$
\mathbb{T}_0^j=\{p\in \M\mid d(p,\tilde{X}_0^j)<\min_{i\neq j}d(p,\tilde{X}_0^i)\},
$$
cf. Remark \ref{suma}. There is a relatively compact neighbourhood $a\in U\subset\mathbb{T}_0^i$. Then by the convergence $(\dag)$ we have $U\cap \mathbb{T}_t^i\neq\varnothing$, for all $t$ close to 0. 
On the other hand, $a_\nu\in U \cap  \tilde{C}_{t_\nu}$ for large $\nu$, which
means that for any such $\nu$, we can find an index $j(\nu)\neq i$ such that
$U \cap \overline{\mathbb{T}^{j(\nu)}_t} \neq\varnothing$, cf. Remark \ref{suma}. As earlier, by extracting a subsequence, we can
choose $j(\nu) = j$ independent of $\nu$, for all $\nu$ large enough.
Then again, by $(\dag)$, the intersection $U\cap \overline{\mathbb{T}}_0^j$ is nonempty which contradicts $U\subset\mathbb{T}_0^i\subset\M\setminus\overline{\mathbb{T}}_0^j$, since $j\neq i$. Note that what plays an important role here is that the cardinality $k_0\geq 2$ (it cannot be one, since by assumptions, the upper limit is nonempty, whence $\tilde{C}_t\neq\varnothing$) is the same for each section, the limit one included.
\end{proof}
\begin{rem}
As a consequence of the Theorem above, we see that in the case when $\#\mathcal{X}_t=k$ is constant and in the limit none of the sets disappear nor merges with another, we have the convergence of the conflict sets $C_t\stackrel{K}{\longrightarrow}C_0$.
\end{rem}

\section{Appendix}

In this last section we prove two additional Propositions concerning the definable case which are of some interest on their own. Here \textit{definable} means \textit{definable in some o-minimal structure} --- we refer the reader to \cite{C}, but also \cite{DD2} in order to see how this is related to subanalytic sets. 

We shall need the following simple Lemma.
\begin{lem}\label{lem1}
 Let $I\subset {\Rz}^k\times{\Rz}$ be a definable set in the variables $(t,s)$. Then the function $r\colon {\Rz}^k\ni t\mapsto \sup I_t\in\overline{\Rz}={\Rz}\cup\{-\infty,+\infty\}$ is definable, too.
\end{lem}
\begin{proof}
The level set $\{t\mid r(t)=-\infty\}$ coincides with ${\Rz}\setminus \pi(I)$ where $\pi(t,s)=t$.  We prove the definability of $r$ over $\pi(I)$.

For $R<+\infty$,
$$
r(t)=R\ \Leftrightarrow\ (\forall s\in I_t, s\leq R)\ \textrm{and}\ (\forall \varepsilon>0, \exists s\in I_t\colon R-\varepsilon<s)
$$
and
$$
r(t)=+\infty\ \Leftrightarrow\ \forall R>0, \exists s\in I_t\colon s\geq R.
$$
The latter shows that the level set $\{t\mid r(t)=+\infty\}$ is definable and the former proves that over its complement $r(t)$ is a definable function, which accounts for the definability of $r(t)$.
\end{proof}
We shall now consider the situation from the proof of Theorem \ref{main} together with the notation introduced there:
\begin{prop}
Assume that $X\subset{\Rz}^k_t\times{\Rz}_x^n$ is definable and has closed $t$-sections for $t\in\pi(X)=:T$ where $\pi(t,x)=t$. Let $m_t(x)$ denote the set of closest points to $x$ in $X_t$. Fix $a\in {\Rz}^n$ and assume that in a neighbourhood $U\ni a$ all the $m_t$'s are univalent for $t\in T$. Put $d_t(x):=\mathrm{dist}(x,X_t)$ and define
$$
r(t):=\sup\{s\geq 1\mid \mathbb{B}(m_t(a)+s(a-m_t(a)),s\cdot d_t(a))\cap X_t=\varnothing\},\> t\in T.
$$
Then the function $r\colon T\to [1,+\infty]$ is definable.
\end{prop}
\begin{proof}
It follows from the description and the definability of the family $X_t$. Indeed,  consider the complement of the set (for a fixed $t\in T$)
$$\{s\geq 1\mid \mathbb{B}(m_t(a)+s(a-m_t(a)),sd_t(a))\cap X_t=\varnothing\}$$
--- as it is described by the condition $$\exists x\in \mathbb{B}(m_t(a)+s(a-m_t(a)),sd_t(a))\cap X_t,$$ it can be written as the image 
$$
I_t:=p(\{(x,s)\in{\Rz}^n\times[1,+\infty)\mid x\in \mathbb{B}(m_t(a)+s(a-m_t(a)),sd_t(a))\cap X_t\})
$$
under the projection $p(x,s)=s$. Now, let us introduce the sets $Y=X\times [1,+\infty)$ and 
$$
B=\{(t,x,s)\in T\times{\Rz}^n\times [1,+\infty)\mid ||x-m_t(a)-s(a-m_t(a))||<sd_t(a)\}.
$$
Both are definable subsets of ${\Rz}^k\times{\Rz}^n\times {\Rz}$: in the case of $Y$ it is obvious and for $B$ it follows from the definability of the maps associating to $t$ the number $d_t(a)$ and the vector $m_t(a)$, respectively (cf. \cite{D}). It remains to observe that 
$$
I_t=p((B\cap Y)_t)
$$
where $(B\cap Y)_t=B_t\cap Y_t$ denotes the $t$-section. Finally, consider the projection $\varrho(t,x,s)=(t,s)$. Then 
$$
p((B\cap Y)_t)=(\varrho(B\cap Y))_t
$$
which shows that the set
$$
I:=\{(t,s)\in T\times [1,+\infty)\mid s\in I_t\}=\varrho(B\cap Y)
$$
is definable and $I_t$ is its $t$-section. 

It remains to use Lemma \ref{lem1} to conclude that $r(t)$ is definable.
\end{proof}

The next Proposition completes the previous Remark \ref{5.9}. Namely, we keep the Setting \ref{konflikt} but now for a definable family of sets $\mathcal{X}_t$ with a multidimensional parameter $t\in T$, i.e. we consider $p\geq 2$ closed, definable, pairwise disjoint sets $X^1,\dots,X^p\subset{\Rz}^k_t\times{\Rz}^n_x$ and $T$ is their common projection with $0\in\overline{T}\setminus T$. Then for the disjoint union $X:=\bigcup_{i=1}^p X^i$, $X_t\stackrel{K}{\rightarrow} \bigcup_{i=1}^p X_0^i=:X_0$ and $X_0$ is definable, since each of the limits $X_0^i$ is such (cf. \cite{DD} Theorem 2.5). Next, $\tilde{X}:=X\cup (\{0\}\times X_0)$ is a definable set with $\tilde{X}_t\stackrel{K}{\rightarrow} \tilde{X}_0$. 

Now, after the regrouping discussed in Remark \ref{5.9}, keeping the notation $C_t$ for the conflict set $\mathrm{Conf}(\mathcal{X}_t)$ and $\tilde{C}_t:=\mathrm{Conf}(\tilde{\mathcal{X}}_t)$, for $t\in T$, we put $C_0=\mathrm{Conf}(\tilde{\mathcal{X}}_0)$ where $\tilde{\mathcal{X}}_0:= \mathcal{X}_0\setminus\{\varnothing\}$ just as earlier, with the convention that $C_0=\varnothing$, when $\#\tilde{\mathcal{X}}_0\leq 1$. Observe that all these sets are definable.

\begin{prop}\label{dfnsc}
The sets $C:=\bigcup_{t\in T}\{t\}\times C_t$ and $\tilde{C}=\bigcup_{t\in T}\{t\}\times\tilde{C}_t$ are definable.
\end{prop}
\begin{proof}
The function $\varrho(t,x)=\min_{i=1}^p d(x,X_t^i)$ is definable which implies the definability of the sets 
$$
\overline{\mathbb{T}}^i=\{(t,x)\in T\times{\Rz}^n\mid d(x, X_t^i)\leq \varrho(t,x)\},\quad i=1,\dots,p
$$
and therefore the definability of the set $\bigcup_{1\leq i<j\leq p}(\overline{\mathbb{T}}^i\cap\overline{\mathbb{T}}^j)$. The latter coincides with $C$ since its $t$-sections are precisely the sets $C_t$ (\footnote{Note that $(A*B)_t=A_t* B_t$, for $*\in\{\cup,\cap\}$.}), cf. Remark \ref{suma}. Of course the same argument works for $\tilde{C}$.
\end{proof}

\section{Statements and declarations}
On behalf of all authors, the corresponding author states that there is no conflict of interest and there are no datasets associated to the study.

\end{document}